\documentclass{amsart}
\usepackage[breaklinks,colorlinks,linkcolor=blue,citecolor=blue,urlcolor=blue]{hyperref}
\usepackage{amssymb}
\usepackage{amsmath}
\usepackage{graphicx}
\graphicspath{ {images/} }

\newtheorem{theorem}{Theorem}[section]

\newtheorem{corollary}[theorem]{Corollary}

\newtheorem{definition}[theorem]{Definition}

\numberwithin{equation}{section}
\numberwithin{equation}{section}

\begin{document}

\title[Mildly version of Hurewicz basis covering property...]{Mildly version of Hurewicz basis covering property and Hurewicz measure zero spaces}

\author{Manoj Bhardwaj}

\address{$^{1}$Department of Mathematics, University of Delhi, New Delhi-110007, India}
\email{manojmnj27@gmail.com}
\thanks{The first author acknowledges the fellowship grant of University Grant Commission, India.}

\author{Alexander V. Osipov}

\address{$^2$Krasovskii Institute of Mathematics and Mechanics, Ural Federal
 University, Ural State University of Economics, Yekaterinburg, Russia}
\email{OAB@list.ru}


\dedicatory{}


\begin{abstract} In this paper, we introduced the mildly version
of the Hurewicz basis covering property, studied by Babinkostova,
Ko\v{c}inac, and Scheepers. A space $X$ is said to have
\textit{mildly-Hurewicz property} if for each sequence $\langle
\mathcal{U}_n : n \in \omega \rangle$ of clopen covers of $X$
there is a sequence $\langle \mathcal{V}_n : n \in \omega \rangle$
such that for each $n$, $\mathcal{V}_n$ is a finite subset of
$\mathcal{U}_n$ and for each $x \in X$, $x$ belongs to $\bigcup
\mathcal{V}_n$ for all but finitely many $n$. Then we
characterized mildly-Hurewicz property by mildly-Hurewicz Basis
property and mildly-Hurewicz measure zero property for metrizable
spaces.
\end{abstract}

\maketitle

\section{Introduction}\label{sec1}

The study of topological properties via various changes is not a
new idea in topological spaces. The study  of selection principles
in topology and their relations to game theory and Ramsey theory
was started by Scheepers \cite{H1} (see also \cite{H2}). In the
last two decades it has gained the enough importance to become one
of the most active areas of set theoretic topology. In covering
properties, Hurewicz property is one of the most important
property.  In 1925, Hurewicz \cite{H4} (see also \cite{H5})
introduced Hurewicz property in topological spaces. This property
is stronger than Lindel$\ddot{o}$f and weaker than
$\sigma$-compactness. In 2001, Ko\v{c}inac \cite{H7}(see also
\cite{H8}) introduced weakly Hurewicz property as a generalization
of Hurewicz spaces. In 2004, the authors Bonanzinga, Cammaroto,
Ko\v{c}inac \cite{H13} introduced the star version of Hurewicz
property and also introduced relativization of strongly
star-Hurewicz property. Every Hurewicz space is weakly Hurewicz.
Continuing this, in 2013, the authors Song and Li \cite{H9}
introduced and studied almost Hurewicz property in topological
spaces. In 2016, Ko\v{c}inac \cite{H6} introduced and studied the
notion of mildly Hurewicz property.

This paper is organized as follows. In section 2, the definitions
of the terms used in this paper are provided. In section 3, mildly
Hurewicz property is characterized using Hurewicz basis property.
In section 4, mildly Hurewicz property is characterized using
mildly Hurewicz measure zero property.

\section{Preliminaries}\label{sec2}

Let $(X,\tau)$ or $X$ be a topological space. We will denote by
$Cl(A)$ and $Int(A)$ the closure  of $A$ and the interior of $A$,
for a subset $A$ of $X$, respectively. The cardinality of a set
$A$ is denoted by $|A|$. Let $\omega$ be the first infinite
cardinal and $\omega_1$ the first uncountable cardinal. As usual,
a cardinal is the initial ordinal and an ordinal is the set of
smaller ordinals. Every cardinal is often viewed as a space with
the usual order topology. For the terms and symbols that we do not
define follow \cite{H10}. The basic definitions are given.

Let $\mathcal{A}$ and $\mathcal{B}$ be collections of open covers of a topological space $X$.

The symbol $S_1(\mathcal{A}, \mathcal{B})$ denotes the selection
hypothesis that for each sequence $\langle \mathcal{U}_n : n \in
\omega \rangle$ of elements of $\mathcal{A}$ there exists a
sequence $\langle U_n : n \in \omega \rangle$ such that for each
$n$, $U_n \in \mathcal{U}_n$ and $\{U_n : n \in \omega\} \in
\mathcal{B}$ (see \cite{H1}).

The symbol $S_{fin}(\mathcal{A}, \mathcal{B})$ denotes the
selection hypothesis that for each sequence  $\langle
\mathcal{U}_n : n \in \omega \rangle$ of elements of $\mathcal{A}$
there exists a sequence $\langle \mathcal{V}_n : n \in \omega
\rangle$ such that for each $n$, $\mathcal{V}_n$ is a finite
subset of $\mathcal{U}_n$ and $\bigcup_{n \in \omega}
\mathcal{V}_n$ is an element of $\mathcal{B}$ (see \cite{H1}).

In this paper $\mathcal{A}$ and $\mathcal{B}$ will be collections of the following open covers of a space $X$:

$\mathcal{O}$ : the collection of all open covers of $X$.

$\mathcal{C}_\mathcal{O}$ : the collection of all clopen covers of $X$.

$\Omega$ : the collection of $\omega$-covers of $X$. An open cover
$\mathcal{U}$ of $X$ is an {\it $\omega$-cover} \cite{H27} if $X$
does not belong to $\mathcal{U}$ and every finite subset of $X$ is
contained in an element of $\mathcal{U}$.

$\mathcal{C}_\Omega$ : the collection of clopen $\omega$-covers of
$X$. A clopen cover $\mathcal{U}$ of $X$ is a clopen
$\omega$-cover if $X$ does not belong to $\mathcal{U}$ and every
finite subset of $X$ is contained in an element of $\mathcal{U}$.

$\Lambda$ : the collection of large covers ($\lambda$-covers) of
$X$. An open cover $\mathcal{U}$ of $X$ is {\it large} (a
$\lambda$-cover) if each $x \in X$ belongs to infinitely many
elements of $\mathcal{U}$.

$\mathcal{C}_\Lambda$ : the collection of clopen large covers
(clopen $\lambda$-covers) of $X$. A clopen cover $\mathcal{U}$ of
$X$ is large (a $\lambda$-cover) if each $x \in X$ belongs to
infinitely many elements of $\mathcal{U}$.

$\Gamma$ : the collection of $\gamma$-covers of $X$. An open cover
$\mathcal{U}$ of $X$ is a {\it $\gamma$-cover} \cite{H27} if it is
infinite and each $x \in X$ belongs to all but finitely many
elements of $\mathcal{U}$.

$\mathcal{C}_\Gamma$ : the collection of clopen $\gamma$-covers of
$X$. A clopen cover $\mathcal{U}$ of $X$ is a clopen
$\gamma$-cover if it is infinite and each $x \in X$ belongs to all
but finitely many elements of $\mathcal{U}$.

$\mathcal{O}^{gp}$ : the collection of groupable open covers. An open cover $\mathcal{U}$ of $X$ is groupable \cite{H28} if it can be expressed as a countable union of finite, pairwise
disjoint subfamilies $\mathcal{U}_n$, such that each $x \in X$ belongs to $\bigcup \mathcal{U}_n$ for all but finitely many $n$.

$\mathcal{C}_\mathcal{O}^{gp}$ : the collection of groupable
clopen covers. A clopen cover $\mathcal{U}$ of $X$ is groupable if
it can be expressed as a countable union of finite, pairwise
disjoint subfamilies $\mathcal{U}_n$, such that each $x \in X$
belongs to $\bigcup \mathcal{U}_n$ for all but finitely many $n$.

$\Omega^{gp}$ : the collection of groupable $\omega$-covers. An
$\omega$-cover $\mathcal{U}$ of $X$ is groupable if it can be
expressed as a countable union of finite, pairwise disjoint
subfamilies $\mathcal{U}_n$, such that each finite subset $F
\subseteq X$ is contained  in $\bigcup \mathcal{U}_n$ for all but
finitely many $n$.

$\mathcal{C}_\Omega^{gp}$ : the collection of groupable clopen
$\omega$-covers. A clopen $\omega$-cover $\mathcal{U}$ of $X$ is
groupable if it can be expressed as a countable union of finite,
pairwise disjoint subfamilies $\mathcal{U}_n$, such that each
finite subset $F \subseteq X$ is contained in $\bigcup
\mathcal{U}_n$ for all but finitely many $n$.

\begin{definition} \label{2.1} \cite{H4}
A space $X$ is said to have \textit{Hurewicz property} (in short $H$) if for each sequence $\langle \mathcal{U}_n : n \in \omega \rangle$ of open covers of $X$ there is a sequence $\langle \mathcal{V}_n : n \in \omega \rangle$ such that for each $n$, $\mathcal{V}_n$ is a finite subset of $\mathcal{U}_n$ and each $x \in X$ belongs to $\bigcup \mathcal{V}_n$ for all but finitely many $n$.
\end{definition}

\begin{definition} \label{2.2} \cite{H6}
A space $X$ is said to have \textit{mildly Hurewicz property} (in short $MH$) if for each sequence $\langle \mathcal{U}_n : n \in \omega \rangle$ of clopen covers of $X$ there is a sequence $\langle \mathcal{V}_n : n \in \omega \rangle$ such that for each $n$, $\mathcal{V}_n$ is a finite subset of $\mathcal{U}_n$ and each $x \in X$ belongs to $\bigcup \mathcal{V}_n$ for all but finitely many $n$.
\end{definition}

For a subset $A$ of a space $X$ and a collection $\mathcal{P}$ of subsets of
$X$, $St(A, \mathcal{P})$ denotes the star of $A$ with respect to $\mathcal{P}$, that is the set $\bigcup \{ P \in \mathcal{P} : A \cap P \neq \emptyset\}$; for $A = \{x\}$, $x \in X$, we write $St(x,\mathcal{P})$ instead of $St(\{x\}, \mathcal{P})$.

\begin{definition} \cite{H13}
A space $X$ is said to have \textit{star-Hurewicz property} (in short $SH$) if for each sequence $\langle \mathcal{U}_n : n \in \omega \rangle$ of open covers of $X$ there is a sequence $\langle \mathcal{V}_n : n \in \omega \rangle$ such that for each $n$, $\mathcal{V}_n$ is a finite subset of $\mathcal{U}_n$ and each $x \in X$ belongs to $St(\bigcup \mathcal{V}_n, \mathcal{U}_n)$ for all but finitely many $n$.
\end{definition}

Let $\mathcal{A}$ and $\mathcal{B}$ be families of subsets of the infinite set $S$. Then $CDR_{sub}(\mathcal{A},\mathcal{B})$ \cite{H1} denotes the statement that for each sequence $\langle A_n : n \in \omega \rangle$ of elements of $\mathcal{A}$ there is a sequence $\langle B_n : n \in \omega \rangle$ such that for each $n$, $B_n \subseteq A_n$, for $m \neq n$, $B_m \cap B_n = \emptyset$, and each $B_n$ is a member of $\mathcal{B}$.

\begin{theorem} \label{64}
If a space $X$ has mildly Hurewicz property and $CDR_{sub}(\mathcal{C}_\mathcal{O},\mathcal{C}_\mathcal{O})$ holds for $X$, then $S_{fin}(\mathcal{C}_\mathcal{O}, \mathcal{C}_\mathcal{O}^{gp})$ holds.
\end{theorem}
\begin{proof}
Let $\langle \mathcal{U}_n : n \in \omega \rangle$ be a sequence
of clopen covers of $X$. Since $X$ has the property
$CDR_{sub}(\mathcal{C}_\mathcal{O},\mathcal{C}_\mathcal{O})$ we
may assume that $\mathcal{U}_n$'s are pairwise disjoint. Since $X$
has mildly Hurewicz property there is a sequence $\langle
\mathcal{V}_n : n \in \omega \rangle$ such that for each $n$,
$\mathcal{V}_n$ is a finite subset of $\mathcal{U}_n$ and for each
$x \in X$, $x$ belongs to $\bigcup \mathcal{V}_n$ for all but
finitely many $n$. Then $\mathcal{V}_n$'s are pairwise disjoint
and hence $\bigcup_{n \in \omega} \mathcal{V}_n$ is a groupable
clopen cover of $X$.
\end{proof}

\section{Mildly Hurewicz Basis property}

In 1924 \cite{H47}, Menger defined the following basis property :

A metric space $(X,d)$ is said to have {\it Menger basis covering
property} if for each basis $\mathcal{B}$ of metric space $(X,d)$
there is a sequence $\langle U_n : n \in \omega \rangle$ of
elements of $\mathcal{B}$ such that $\{U_n : n \in \omega\}$ is a
cover of $X$ and $lim_{n \rightarrow \infty} diam_d(U_n) = 0$.

In 1925 \cite{H4}, Hurewicz showed that the Menger basis property is equivalent to the Menger covering property $S_{fin}(\mathcal{O},\mathcal{O})$.

In 2004 \cite{H42}, Babinkostova, Ko$\check{c}$inac and Scheepers defined the following basis property :

A metric space $(X,d)$ is said to have {\it Hurewicz basis
covering property} \cite{H42} if for each basis $\mathcal{B}$ of
metric space $(X,d)$ there is a sequence $\langle U_n : n \in
\omega \rangle$ of elements of $\mathcal{B}$ such that $\{U_n : n
\in \omega\}$ is a groupable cover of $X$ and $lim_{n \rightarrow
\infty} diam_d(U_n) = 0$.

Recall that a metric space is {\it crowded} if it does not have
isolated points.

\begin{theorem} \cite{H42}
For a crowded metric space $(X,d)$, $X$ has Hurewicz property if
and only if it has Hurewicz basis property.
\end{theorem}

Since Hurewicz property and star-Hurewicz property are equivalent
in metrizable spaces, it can be noted that for a metric crowded
space $(X,d)$, $X$ has star-Hurewicz property if and only if it
has Hurewicz basis property.

In 2020 \cite{H42a}, Bhardwaj and Osipov defined the following basis property :

A metric space $(X,d)$ is said to have star-Hurewicz basis property if for each basis $\mathcal{B}$ of metric space $(X,d)$, there is a sequence $\langle V_n : n \in \omega \rangle$ of elements of $\mathcal{B}$ such that $\{St(V_n, \mathcal{B}) : n \in \omega\}$ is a groupable cover of $X$ and $lim_{n \rightarrow \infty} diam_d(V_n) = 0$.

\begin{theorem} \cite{H42a}
Let $(X,d)$ be a crowded metric space. The followings are
equivalent :
\begin{enumerate}
\item $X$ has star-Hurewicz property; \item for each basis
$\mathcal{B}$ of metric space $(X,d)$ and for each sequence
$\langle \mathcal{U}_n : n \in \omega \rangle$ of open covers of
$(X,d)$, there is a sequence $\langle \mathcal{V}_n : n \in \omega
\rangle$ of finite sets of elements of $\mathcal{B} \wedge
\mathcal{U}_n=\{B\cap U: B\in \mathcal{B}, U\in\mathcal{U}_n\}$
such that $\{St(\bigcup \mathcal{V}_n, \mathcal{B} \wedge
\mathcal{U}_n) : n \in \omega\}$ is a groupable cover of $X$ and
$lim_{n \rightarrow \infty} diam_d(U_n) = 0$ for $U_n \in
\mathcal{V}_n$; \item for each basis $\mathcal{B}$ of metric space
$(X,d)$ and for each sequence $\langle \mathcal{U}_n : n \in
\omega \rangle$ of open covers of $(X,d)$, there is a sequence
$\langle \mathcal{V}_n : n \in \omega \rangle$ of finite sets of
elements of $\mathcal{B}$ such that $\{St(\bigcup \mathcal{V}_n,
\mathcal{B} \wedge \mathcal{U}_n) : n \in \omega\}$ is a groupable
cover of $X$ and $lim_{n \rightarrow \infty} diam_d(U_n) = 0$ for
$U_n \in \mathcal{V}_n$.
\end{enumerate}
\end{theorem}


Now we define a mildly version of Hurewicz basis property.

\begin{definition}
A metric space $(X,d)$ is said to have mildly-Hurewicz basis
property if for each basis $\mathcal{B}$  consisting of clopen
sets of metric space $(X,d)$ there is a sequence $\langle U_n : n
\in \omega \rangle$ of elements of $\mathcal{B}$ such that $\{U_n
: n \in \omega\}$ is a groupable clopen cover of $X$ and $lim_{n
\rightarrow \infty} diam_d(U_n) = 0$.
\end{definition}

\begin{theorem}
If $(X,d)$ is a crowded metric space for which $CDR_{sub}(\mathcal{C}_\mathcal{O}, \mathcal{C}_\mathcal{O})$ holds, then following statements are equivalent:
\begin{enumerate}
\item $X$ has mildly-Hurewicz property; \item $X$ has
mildly-Hurewicz basis property.
\end{enumerate}
\end{theorem}
\begin{proof}
Let $X$ has mildly-Hurewicz property and let $\mathcal{B}$ be a
basis of $X$ consisting clopen sets. Now define $\mathcal{U}_n =
\{U \in \mathcal{B} : diam_d(U) < 1/(n+1) \}$. Then for each $n$,
$\mathcal{U}_n$ is a large clopen cover of $X$. Since $X$ has
mildly-Hurewicz property, by Theorem \ref{64}, there is a sequence
$\langle \mathcal{V}_n : n \in \omega \rangle$ such that for each
$n$, $\mathcal{V}_n$ is a finite subset of $\mathcal{U}_n
\subseteq \mathcal{B}$ and $\bigcup_{n \in \omega} \mathcal{V}_n$
is a groupable clopen cover of $X$. Then for $\bigcup_{n \in
\omega} \mathcal{V}_n = \{U_n : n \in \omega\}$, $lim_{n
\rightarrow \infty} diam_d(U_n) = 0$.

Conversely, let $X$ be a space having mildly-Hurewicz basis property and $\langle \mathcal{U}_n : n \in \omega \rangle$ be a sequence of clopen covers of $X$. Now assume that if a clopen set $V$ is a subset of an element of $\mathcal{U}_n$, then $V \in \mathcal{U}_n$.
For each $n$ define
\begin{center}
$\mathcal{H}_n = \{U_1 \cap U_2 \cap ...\cap U_n : (\forall i \leq n)(U_i \in \mathcal{U}_i)\} \setminus \{\emptyset\}$.
\end{center}
Then for each $n$, $\mathcal{H}_n$ is a clopen cover of $X$ and has the property that if a clopen set $V$ is a subset of an element of $\mathcal{H}_n$, then $V \in \mathcal{H}_n$.

Now let $\mathcal{U}$ be the set $\{U \cup V : (\exists n)(U,V \in \mathcal{H}_n$ and $diam_d(U \cup V) > 1/n)\}$. First we show that $\mathcal{U}$ is a basis for $X$ consisting clopen sets. For it, let $W$ be an open subset containing a point $x$. Since $(X,d)$ does not have isolated points, $x$ is not an isolated point of $X$. Then we can choose $y \in W \setminus \{x\}$ and $n > 1$ with $d(x,y) > 1/n$. Since $\mathcal{H}_n$ is a clopen cover of $X$, there are $U^{'},V^{'} \in \mathcal{H}_n$ such that $x \in U^{'}$ and $y \in V^{'}$. Now put $U = U^{'} \cap W \setminus \{y\}$ and $V = V^{'} \cap W \setminus \{x\}$. Then $U,V \in \mathcal{H}_n$ since $\mathcal{H}_n$ has the property that if a clopen set $V$ is a subset of an element of $\mathcal{H}_n$, then $V \in \mathcal{H}_n$. Also $U \cup V \subseteq W$ and $diam_d(U \cup V) \geq d(x,y) > 1/n$. So $U \cup V \in \mathcal{U}$ and $x \in U \cup V \subseteq W$. Thus $\mathcal{U}$ is a basis for $X$ consisting clopen sets.

Since $X$ has mildly-Hurewicz basis property, there is a sequence $\langle W_n : n \in \omega \rangle$ of elements of $\mathcal{U}$ such that $\{W_n : n \in \omega\}$ is a groupable clopen cover of $X$ and $lim_{n \rightarrow \infty} diam_d(W_n) = 0$. Then $\bigcup_{n \in \omega} \mathcal{W}_n = \{W_n : n \in \omega\}$ such that each $\mathcal{W}_n$ is finite, $\mathcal{W}_n \cap \mathcal{W}_m = \emptyset$ for $n \neq m$ and each $x \in X$, $x$ belongs to $\bigcup \mathcal{W}_n$ for all but finitely many $n$. Without loss of generality, let
\begin{center}
$\mathcal{W}_1 = \{W_1,W_2,...,W_{m_1 - 1}\}$; \\
$\mathcal{W}_2 = \{W_{m_1},W_{m_1 + 1},...,W_{m_2 - 1}\}$; \\
.\\
.\\
.\\
$\mathcal{W}_n = \{W_{m_{n-1}},W_{m_{n-1} + 1},...,W_{m_n}\}$; \\
\end{center}
and so on. Now we get a sequence $m_1 < m_2 < m_3 < ...< m_k <...$
obtained from groupability  of  $\{W_n : n \in \omega\}$ such that
for each $x \in X$, for all but finitely many $k$ there is a $j$
with $m_{k-1} \leq j < m_k$ such that $x \in \mathcal{W}_j$.

Since $W_n \in \mathcal{U}$, so there is $k_n$ such that $U_n,V_n \in \mathcal{H}_{k_n}$ and $W_n = U_n \cup V_n$ with $diam_d(W_n) > 1/k_n$. For each $n$, select the least $k_n$ and sets $U_n$ and $V_n$ from $\mathcal{U}_{k_n}$. Since each $\mathcal{U}_n$ has the property that if a clopen set $V$ is a subset of an element of $\mathcal{U}_n$, then $V \in \mathcal{U}_n$, $U_n,V_n \in \mathcal{U}_{k_n}$. Since $lim_{n \rightarrow \infty} diam_d(W_n) = 0$, for each $W_n$, there is maximal $m_n$ such that $diam_d(W_n) < 1/m_n$. Then $1/k_n < diam_d(W_n) < 1/m_n$ implies that $k_n > m_n$ for each $n$ and $lim_{n \rightarrow \infty} m_n = \infty$. Since $lim_{n \rightarrow \infty} diam_d(W_n) = 0$, so for each $k_n$, there are only finitely many $W_n$ for which the representatives $U_n,V_n$ are from $\mathcal{U}_{k_n}$ and have $diam_d(U_n \cup V_n) > 1/k_n$. Let $\mathcal{V}_{k_n}$ be the finite set of such $U_n,V_n$.

Now choose $l_1 > 1$ so large such that each $W_i$ with $i \leq m_1$ has a representation of the form $U \cup V$ and $U$'s  and $V$'s are from the sets $\mathcal{V}_{k_i}, k_i \leq l_1$.
Then select $j_1$ so large such that for all $i > j_1$, if $W_i$ has representatives from $\mathcal{V}_{k_i}$, then $k_i > l_1$.

For choosing $l_2$, let $m_k$ be the smallest greater than $j_1$,
and now choose $l_2$ so large that if $W_i$ with $m_k \leq i
<m_{k+1}$ uses a $\mathcal{V}_{k_i}$, then $k_i \leq l_2$, that
is, choose maximal of $k_i$ for which $m_k \leq i <m_{k+1}$ and
say $l_2$, then $l_1 < k_i \leq l_2$. Now choose maximal of $i$
for which the representation of $W_i$ from $\mathcal{V}_{k_i}$
where $l_1 < k_i \leq l_2$ and say $j_2$, then $j_2 > j_1$ and
$\forall i \geq j_2$ if $W_i$ uses $\mathcal{V}_{k_n}$, then $k_n
> l_2$.

Similarly alternately choose $l_m$ and $j_m$. For each $m$ if we consider the least $m_k > l_m$, then :
\begin{enumerate}
\item
if $W_i$ with $m_k \leq i <m_{k+1}$ uses a $\mathcal{V}_{k_i}$ then $l_m < k_i \leq l_{m+1}$;
\item
if $i \geq j_m$ then if $W_i$ uses $\mathcal{V}_{k_n}$ then $k_n > l_m$.
\end{enumerate}

For each $V \in \mathcal{V}_{k_n}$ with $k_n \leq l_1$, $V \in \mathcal{H}_{k_n}$ and $V \in \mathcal{U}_{k_i}$ for each $k_i \leq k_n$. Then $V \in \mathcal{U}_1$ and let $\mathcal{G}_1$ be collection of such $V \in \mathcal{V}_{k_n}$ with $k_n \leq l_1$. Then $\mathcal{G}_1 \subseteq \mathcal{U}_1$ is a finite subset.

Now for each $V \in \mathcal{V}_{k_n}$ with $l_p < k_n \leq
l_{p+1}$(as $p < l_p$),  $V \in \mathcal{H}_{k_n}$ and $V \in
\mathcal{U}_{k_i}$ for each $k_i \leq k_n$. Then $V \in
\mathcal{U}_{p}$ and let $\mathcal{G}_p$ be collection of such $V
\in \mathcal{V}_{k_n}$ with $l_p < k_n \leq l_{p+1}$. Then
$\mathcal{G}_p$ is a finite subset of $\mathcal{U}_p$.

Then we have that for each $x \in X$, $x \in \bigcup \mathcal{G}_p$ for all but finitely many $p$. It follows that $X$ has the mildly-Hurewicz property.
\end{proof}

In \cite{H6}, it was shown that for a zero dimensional space, Hurewicz and mildly-Hurewicz properties are equivalent. Now we have the following corollary.

\begin{corollary}
If $(X,d)$ is a zero dimensional crowded metric space for which $CDR_{sub}(\mathcal{O}, \mathcal{O})$ holds, then following statements are equivalent:
\begin{enumerate}
\item
$X$ has Hurewicz property;
\item
$X$ has mildly Hurewicz property;
\item
$X$ has star-Hurewicz property;
\item
$X$ has Hurewicz basis property;
\item
$X$ has mildly Hurewicz basis property
\item
for each basis $\mathcal{B}$ of metric space $(X,d)$ and for each sequence $\langle \mathcal{U}_n : n \in \omega \rangle$ of open covers of $(X,d)$, there is a sequence $\langle \mathcal{V}_n : n \in \omega \rangle$ of finite sets of elements of $\mathcal{B} \wedge \mathcal{U}_n$ such that $\{St(\bigcup \mathcal{V}_n, \mathcal{B} \wedge \mathcal{U}_n) : n \in \omega\}$ is a groupable cover of $X$ and $lim_{n \rightarrow \infty} diam_d(U_n) = 0$ for $U_n \in \mathcal{V}_n$;
\item
for each basis $\mathcal{B}$ of metric space $(X,d)$ and for each sequence $\langle \mathcal{U}_n : n \in \omega \rangle$ of open covers of $(X,d)$, there is a sequence $\langle \mathcal{V}_n : n \in \omega \rangle$ of finite sets of elements of $\mathcal{B}$ such that $\{St(\bigcup \mathcal{V}_n, \mathcal{B} \wedge \mathcal{U}_n) : n \in \omega\}$ is a groupable cover of $X$ and $lim_{n \rightarrow \infty} diam_d(U_n) = 0$ for $U_n \in \mathcal{V}_n$.
\end{enumerate}
\end{corollary}

\section{Mildly Hurewicz measure zero property}

Recall that a set of reals $X$ is null (or has measure zero) if for each positive $\epsilon$ there exists a cover $\{I_n\}_{n \in \omega}$ of $X$ such that $\Sigma_n$ diam$(I_n) < \epsilon$.

To generalize the notion of measure zero or null set, in 1919 \cite{H34}, Borel defined a notion stronger than measure zeroness. Now this notion is known as strong measure zeroness or strongly null set.

Borel strong measure zero: $Y$ is Borel strong measure zero if there is for each sequence $\langle \epsilon_n : n \in \omega \rangle$ of positive real numbers a sequence $\langle J_n : n \in \omega \rangle$ of subsets of $Y$ such that each $J_n$ is of diameter $< \epsilon_n$, and $Y$ is covered by $\{J_n : n \in \omega\}$.

But Borel was unable to construct a nontrivial (that is, an uncountable) example of a strongly null set. He therefore conjectured that there exists no such examples.

In 1928, Sierpinski observed that every Luzin set is strongly
null, thus the Continuum Hypothesis implies that Borel's
Conjecture is false.

Sierpinski asked whether the property of being strongly null is
preserved when taking homeomorphic (or even continuous) images.

In 1941, the answer given by Rothberger is negative when the
Continuum Hypothesis.  This lead Rothberger to introduce the
following topological version of strong measure zero (which is
preserved when taking continuous images).

A space $X$ is said to have {\it Rothberger property} if it
satisfies the selection principle $S_1(\mathcal{O}, \mathcal{O})$.

In 1988(\cite{H56}) Miller and Fremlin proved that a space $Y$ has the Rothberger property ( $S_1(\mathcal{O},\mathcal{O})$) if and only if it has Borel strong measure zero with respect to each metric on $Y$ which generates the topology of $Y$.

In \cite{H42}, Hurewicz measure zero property was defined.

Hurewicz measure zero : a metric space $(X,d)$ is {\it Hurewicz
measure zero} if for each sequence $\langle \epsilon_n: n \in
\omega \rangle$ of positive real numbers there is a sequence
$\langle \mathcal{V}_n : n \in \omega \rangle$ such that:
\begin{enumerate}
\item for each $n$, $\mathcal{V}_n$ is a finite set of open
subsets in $X$; \item for each $n$, each member of $\mathcal{V}_n$
has $d$-diameter less than $\epsilon_n$; \item $\bigcup_{n \in
\omega} \mathcal{V}_n$ is a groupable cover of $X$.
\end{enumerate}

\begin{theorem} \cite{H42}
Let $(X,d)$ be a zero-dimensional separable crowded metric space. The following statements are equivalent:
\begin{enumerate}
\item
$X$ has Hurewicz property;
\item
$X$ has Hurewicz measure zero property with respect to every metric on $X$ which gives $X$ the same topology as $d$ does.
\end{enumerate}
\end{theorem}

In 2020 \cite{H42a}, Bhardwaj and Osipov defined the following measure zeroness property :

A metric space $(X,d)$ is {\it star-Hurewicz measure zero} if for
each sequence $\langle \epsilon_n: n \in \omega \rangle$ of
positive real numbers there is a sequence $\langle \mathcal{V}_n :
n \in \omega \rangle$ such that:
\begin{enumerate}
\item for each $n$, $\mathcal{V}_n$ is a finite set of open
subsets of $X$; \item for each $n$, each member of $\mathcal{V}_n$
has $d$-diameter less than $\epsilon_n$; \item $\{St(\bigcup
\mathcal{V}_n,\mathcal{U}_n) : n \in \omega\}$ is a groupable
cover of $X$, where $\mathcal{U}_n = \{U \subset X : U$ is open
set with $diam_d(U) < \epsilon_n\}$ for each $n$.
\end{enumerate}

\begin{theorem} \cite{H42a}
Let $(X,d)$ be a zero-dimensional separable metric space with no isolated points. The following statements are equivalent:
\begin{enumerate}
\item
$X$ has star-Hurewicz property;
\item
$X$ is star-Hurewicz measure zero with respect to every metric which gives $X$ the
same topology as $d$ does.
\end{enumerate}
\end{theorem}

Now we consider the mildly version of Hurewicz measure zero property.

\begin{definition}
A metric space $(X,d)$ is mildly-Hurewicz measure zero if for each sequence $\langle \epsilon_n: n \in \omega \rangle$ of positive real numbers there is a sequence $\langle \mathcal{V}_n : n \in \omega \rangle$ such that:
\begin{enumerate}
\item
for each $n$, $\mathcal{V}_n$ is a finite set of clopen subsets in $X$;
\item
for each $n$, each member of $\mathcal{V}_n$ has d-diameter less than $\epsilon_n$;
\item
$\bigcup_{n \in \omega} \mathcal{V}_n$ is a groupable clopen cover of $X$.
\end{enumerate}
\end{definition}

\begin{theorem}
If $(X,d)$ is a zero-dimensional separable crowded metric space for which $CDR_{sub}(\mathcal{C}_\mathcal{O},\mathcal{C}_\mathcal{O})$ holds, then following statements are equivalent:
\begin{enumerate}
\item $X$ has mildly-Hurewicz property; \item $X$ is
mildly-Hurewicz measure zero with respect to every metric on $X$
which gives $X$ the same topology as $d$ does.
\end{enumerate}
\end{theorem}
\begin{proof}
For $(1) \Rightarrow (2)$, let $X$ has mildly-Hurewicz property
and let $\langle \epsilon_n : n \in \omega \rangle$ be a sequence
of positive real numbers. For each $n$, define $\mathcal{U}_n =
\{U \subset X : U$ is clopen set with $diam_d(U) < \epsilon_n\}$.
Since $X$ is a zero-dimensional metric space, $\mathcal{U}_n$ is a
large clopen cover of $X$ for each $n$. Since $X$ has
mildly-Hurewicz property, by Theorem \ref{64}, there is a sequence
$\langle \mathcal{V}_n : n \in \omega \rangle$ such that for each
$n$, $\mathcal{V}_n$ is a finite subset of $\mathcal{U}_n$ and
$\bigcup \mathcal{V}_n\in \mathcal{C}_{\mathcal{O}}^{gp}$. Hence
$X$ has mildly-Hurewicz measure zero with respect to every metric
on $X$ which gives $X$ the same topology as $d$ does.

For $(2) \Rightarrow (1)$, let $d$ be an arbitrary metric on $X$
which gives $X$  the same topology as the original one. Let
$\langle \mathcal{U}_n : n \in \omega \rangle$ be a sequence of
clopen covers of $X$. Since $X$ is a zero-dimensional metric
space, replace $\mathcal{U}_n$ by $\{U \subseteq X : U$ clopen,
$diam_d(U) < 1/n$ and $\exists V \in \mathcal{U}_n$ such that $U
\subseteq V \}$ for each $n$. Also $X$ is a separable metric
space, replace last cover by a countable subcover $\{U_m : m
\in\omega\}$. Since the cover is countable and sets are clopen, it
can be made disjoint clopen cover refining $\mathcal{U}_n$ for
each $n$. Also for each $n$, each member of this new cover has
$diam_d \leq 1/n$. Also by taking intersections of each new cover
with the last new cover we obtain a new cover. Now name this cover
$\mathcal{U}^\star_n$ for each $n$. So $\langle
\mathcal{U}^\star_n : n \in \omega \rangle$ is a sequence of
clopen covers such that for each $n$:
\begin{enumerate}
\item $\mathcal{U}^\star_n$ is a clopen pairwise disjoint cover of
$X$ refining $\mathcal{U}_n$; \item for each $V \in
\mathcal{U}^\star_n$, $diam_d(V) \leq 1/n$; \item
$\mathcal{U}^\star_{n+1}$ refines $\mathcal{U}^\star_n$.
\end{enumerate}
Now define a metric $d^\star$ on $X$ by $d^\star(x,y) = 1/(n+1)$
where $n$ is the least natural number such that there exist $U \in
\mathcal{U}^\star_n$ with $x \in U$ and $y \notin U$. It can be
easily seen that $d^\star$ generates the same topology on $X$ as
$d$ does. Since $X$ has mildly-Hurewicz measure zero with respect
to $d^\star$, by setting $\epsilon_n = 1/(n+1)$ for each $n$,
there are finite sets $\mathcal{V}_n$ such that $diam_d^\star(U)$
is less than $\epsilon_n( = 1/(n+1))$ whenever $U \in
\mathcal{V}_n$, and $\bigcup_{n \in \omega} \mathcal{V}_n$ is a
groupable clopen cover of $X$.

Let $\langle \mathcal{W}_n : n \in \omega \rangle$ be a sequence of finite subsets of $\bigcup_{n \in \omega} \mathcal{V}_n$ such that $\mathcal{W}_m \cap \mathcal{W}_n = \emptyset$ whenever $m \neq n$, and $\bigcup_{k \in \omega} \mathcal{W}_k = \bigcup_{n \in \omega} \mathcal{V}_n$, and for each $y \in X$, $y \in \bigcup \mathcal{W}_k$ for all but finitely many $k$.

Since $\mathcal{W}_1$ is finite, so choose $i_1$ such that $\mathcal{W}_1 \subseteq \bigcup_{i \leq i_1} \mathcal{V}_i$. Then $\bigcup_{i \leq i_1} \mathcal{V}_i$ is finite and exhausted in a finite number of $\mathcal{W}_k$'s and choose $j_1$ such that for each $i \geq j_1$, if $V \in \mathcal{W}_i$, then $V \notin \bigcup_{i \leq i_1} \mathcal{V}_i$.

Now $\mathcal{W}_{j_1}$ is finite, so choose $i_2 > i_1$ such that $\mathcal{W}_{j_1} \subseteq \bigcup_{i_1 < i \leq i_2} \mathcal{V}_i$. Then $\bigcup_{i_1 < i \leq i_2} \mathcal{V}_i$ is finite and exhausted in a finite number of $\mathcal{W}_k$'s and choose $j_2$ such that for each $i \geq j_2$, if $V \in \mathcal{W}_i$, then $V \notin \bigcup_{i_1 < i \leq i_2} \mathcal{V}_i$.

Alternatively, we choose sequences $1 < i_1 < i_2 < ... < i_m < ...$ and $j_0 = 1 < j_1 < j_2 < ...< j_m < ...$ such that :
\begin{enumerate}
\item
Each element of $\mathcal{W}_1$ belongs to $\bigcup_{i \leq i_1} \mathcal{V}_i$;
\item
For each $i \geq j_k$, if $U \in \mathcal{W}_{j_k}$, then $U \notin \bigcup_{i \leq i_k} \mathcal{V}_i$;
\item
Each element of $\mathcal{W}_{j_k}$ belongs to $\bigcup_{i_k < i \leq i_{k+1}} \mathcal{V}_i$.
\end{enumerate}
Then for each element $V$ of $\mathcal{W}_{j_k}$ has $d^\star$-diameter less than $\epsilon_{i_k} = 1/i_k + 1 \leq 1/k+1$ since $i_k \geq k$. As $V$ is clopen set in $(X,d^\star)$ and $diam_d^\star(V) < 1/k+1$, then $V \subseteq B_d^\star(x,1/k+1)$, where $B_d^\star(x,1/k+1)$ is an open ball centered at $x \in V$ and of radius $1/k+1$ in $(X,d^\star)$. Now $B_d^\star(x,1/k+1) = \{y \in X : d^\star(x,y) < 1/k+1\}$. So for each $y \in B_d^\star(x,1/k+1), d^\star(x,y) < 1/k+1$, there is $U \in \mathcal{U}^\star_n$ such that $x \in U$ and $y \notin U$ for some $n > k$. So for all $k \leq n$, there is no set $U \in \mathcal{U}^\star_k$ such that $x \in U$ and $y \notin U$. Thus for all $k \leq n$, there is a set $U \in \mathcal{U}^\star_k$ such that $x,y \in U$ for all $x,y \in B_d^\star(x,1/k+1)$, that is, $B_d^\star(x,1/k+1) \subseteq U \in \mathcal{U}^\star_k$ for all $k \leq n$.
Thus, by definition of $d^\star$, each element $V$ of $\mathcal{W}_{j_k}$ is a subset of an element of $\mathcal{U}^\star_k$, each of which in turn is a subset of an element of $\mathcal{U}_k$. For each $k$ and for each element $V$ of $\mathcal{W}_{j_k}$, choose a $U \in \mathcal{U}^\star_k$ with $V \subseteq U$  and let $\mathcal{G}_k$ be the finite set of such chosen $U$'s and $\mathcal{G}_k$ is a finite subset of $\mathcal{U}_k$. So for each $k$, $\bigcup \mathcal{W}_{j_k} \subseteq \bigcup \mathcal{G}_k$.

Then we have that for each $x \in X$, $x \in \bigcup
\mathcal{G}_p$ for all but finitely many $p$. It follows that $X$
has mildly-Hurewicz property.
\end{proof}

Now we have the following corollary.

\begin{corollary}
If $(X,d)$ is a zero-dimensional separable crowded metric space for which $CDR_{sub}(\mathcal{O}, \mathcal{O})$ holds, then following statements are equivalent:
\begin{enumerate}
\item $X$ has Hurewicz property; \item $X$ has star-Hurewicz
property; \item $X$ has mildly-Hurewicz property; \item $X$ has
Hurewicz basis property; \item $X$ has mildly-Hurewicz basis
property; \item for each basis $\mathcal{B}$ of metric space
$(X,d)$ and for each sequence $\langle \mathcal{U}_n : n \in
\omega \rangle$ of open covers of $(X,d)$, there is a sequence
$\langle \mathcal{V}_n : n \in \omega \rangle$ of finite sets of
elements of $\mathcal{B} \wedge \mathcal{U}_n$ such that
$\{St(\bigcup \mathcal{V}_n, \mathcal{B} \wedge \mathcal{U}_n) : n
\in \omega\}$ is a groupable cover of $X$ and $lim_{n \rightarrow
\infty} diam_d(U_n) = 0$ for $U_n \in \mathcal{V}_n$; \item for
each basis $\mathcal{B}$ of metric space $(X,d)$ and for each
sequence $\langle \mathcal{U}_n : n \in \omega \rangle$ of open
covers of $(X,d)$, there is a sequence $\langle \mathcal{V}_n : n
\in \omega \rangle$ of finite sets of elements of $\mathcal{B}$
such that $\{St(\bigcup \mathcal{V}_n, \mathcal{B} \wedge
\mathcal{U}_n) : n \in \omega\}$ is a groupable cover of $X$ and
$lim_{n \rightarrow \infty} diam_d(U_n) = 0$ for $U_n \in
\mathcal{V}_n$; \item $X$ is Hurewicz measure zero with respect to
every metric on $X$ which gives $X$ the same topology as $d$ does;
\item $X$ is star-Hurewicz measure zero with respect to every
metric on $X$ which gives $X$ the same topology as $d$ does; \item
$X$ is mildly-Hurewicz measure zero with respect to every metric
on $X$ which gives $X$ the same topology as $d$ does.
\end{enumerate}
\end{corollary}


\begin{thebibliography}{11}

\bibitem{H42} L. Babinkostova, Lj.D.R. Ko$\check{c}$inac and M. Scheepers, \textit{Combinatorics of open covers (VIII)},  Topology Appl.,\textbf{140} (2004), 15-32.


\bibitem{H42a} M. Bhardwaj and A. V. Osipov, \textit{Star versions of Hurewicz basis covering property and strong measure zero spaces}, Turkish J. Math., \textbf{44} (2020), 1042-1053.

\bibitem{H13} M. Bonanzinga, F. Cammaroto and Lj.D.R. Ko$\check{c}$inac, \textit{Star-Hurewicz and related properties}, Appl. Gen. Topol., \textbf{5} (2004), 79-89.

\bibitem{H34} E. Borel, \textit{Sur la classification des ensembles de mesure nulle}, Bull. Soc. Math. France, \textbf{47} (1919), 97-125.


\bibitem{H10} R. Engelking, \textit{General Topology, Revised and completed edition}, Heldermann Verlag Berlin (1989).




\bibitem{H27} J. Gerlits and Zs. Nagy, \textit{Some properties of $C(X), I$}, Topology Appl., \textbf{14} (1982),  151-161.

\bibitem{H4} W. Hurewicz, \textit{$\ddot{U}$ber eine verallgemeinerung des Borelschen Theorems}, Math. Z. \textbf{24} (1925), 401-421.

\bibitem{H5} W. Hurewicz, \textit{$\ddot{U}$ber Folgen stetiger Funktionen}, Fund. Math., \textbf{9} (1927), 193-204.

\bibitem{H2} W. Just, A.W. Miller, M. Scheepers and P.J. Szeptycki, \textit{The combinatorics of open covers (II)}, Topology Appl., \textbf{73} (1996), 241-266.



\bibitem{H6} Lj.D.R. Ko$\check{c}$inac, \textit{On Mildly Hurewicz Spaces}, Int. Math. Forum., \textbf{11}(12) (2016), 573-582.

\bibitem{H7} Lj.D.R. Ko$\check{c}$inac, \textit{The Pixley-Roy topology and selection principles}, Quest. Ans. Gen. Top., \textbf{19} (2001), 219-225.

\bibitem{H28} Lj. D.R. Ko$\check{c}$inac and M. Scheepers, \textit{Combinatorics of open covers (VII): Groupability}, Fund. Math., \textbf{179} (2003),  131-155 .


\bibitem{H47} K. Menger, \textit{Einige Überdeckungssätze der punktmengenlehre}, Sitzungsberischte Abt. 2a, Mathematik, Astronomie, Physik, Meteorologie und Mechanik (Wiener
Akademie, Wien) \textbf{133} (1924), 421-444.


\bibitem{H56} A.W. Miller and D.H. Fremlin, \textit{Some properties of Hurewicz, Menger and Rothberger}, Fund. Math., \textbf{129} (1988), 17-33.



\bibitem{H8} M. Sakai, \textit{The weak Hurewicz property of Pixley-Roy hyperspaces}, Topology Appl., \textbf{160} (2013), 2531-2537.

\bibitem{H1} M. Scheepers, \textit{Combinatorics of open covers (I) : Ramsey theory},  Topology Appl.,\textbf{69} (1996), 31-62.



\bibitem{H9} Y. K. Song and R. Li, \textit{The almost Hurewicz spaces}, Quest. Ans. Gen. Top., \textbf{31} (2013), 131-136.

\end{thebibliography}
\end{document}